\documentclass[
final
]{dmtcs-episciences}

\usepackage{amsmath}
\usepackage{amssymb}
\usepackage[utf8]{inputenc}
\usepackage{subfigure}

\newtheorem{theorem}{Theorem}[section]
\newtheorem{lemma}[theorem]{Lemma}

\newtheorem{conjecture}[theorem]{Conjecture}

\newtheorem{definition}[theorem]{Definition}
\newtheorem{example}[theorem]{Example}

\author{Ruixia Wang\thanks{This work is supported by  the Natural Science Foundation of Shanxi Province (201901D111022)}
  \and Linxin Wu
  \and Wei Meng}
\title[Extremal digraphs on Meyniel-type condition for hamiltonian cycles in balanced bipartite digraphs]{Extremal digraphs on Meyniel-type condition for hamiltonian cycles in balanced bipartite digraphs}
\affiliation{School of Mathematical Sciences, Shanxi University, Taiyuan, Shanxi, 030006, PR China
 }
\received{2019-10-17}
\revised{2020-05-08}
\accepted{2021-08-23}
\publicationdetails{23}{2021}{3}{8}{5851}

\begin{document}
\maketitle
\begin{abstract}
  Let $D$ be a strong balanced digraph on $2a$ vertices.
Adamus et al. have proved that $D$ is hamiltonian if $d(u)+d(v)\ge 3a$ whenever $uv\notin A(D)$ and $vu\notin A(D)$. The lower bound $3a$ is tight. In this paper, we shall show that the extremal digraph on this condition is two classes of digraphs that can be clearly characterized. Moreover, we also show that if $d(u)+d(v)\geq 3a-1$ whenever $uv\notin A(D)$ and $vu\notin A(D)$, then $D$ is traceable. The lower bound $3a-1$ is tight. 
\end{abstract}

\section{Terminology and introduction}

\quad In this paper, we consider finite digraphs without loops and multiple arcs. We shall assume that the reader is familiar with the standard terminology on digraphs and refer the reader to \cite{bang} for terminology not defined here. Let $D$ be a digraph with vertex set $V(D)$ and arc set $A(D)$. For any $x,y\in V(D)$, we will write $x\rightarrow y$ if $xy\in A(D)$, also write $x\leftrightarrow y$ if $x\rightarrow y$ and $y\rightarrow x$. For disjoint subsets $X$ and $Y$ of $V(D)$, $X\rightarrow Y$ means that every vertex of $X$ dominates every vertex of $Y$, $X\Rightarrow Y$ means that there is no arc from $Y$ to $X$ and $X\mapsto Y$ means that both of $X\rightarrow Y$ and $X\Rightarrow Y$ hold. For a vertex set $S\subset V(D)$, we denote by $N^+(S)$ the set of vertices in $V(D)$ dominated by the vertices of $S$; i.e. $N^+(S)=\{u\in V(D): vu\in A(D)\  \mbox{for some}\ v\in S\}.$ Similarly, $N^-(S)$ denotes the set of vertices of $V(D)$ dominating vertices of $S$; i.e.  $N^-(S)=\{u\in V(D): uv\in A(D)\  \mbox{for some}\  v\in S\}.$ If $S=\{v\}$ is a single vertex, the cardinality of $N^+(v)$ (resp. $N^-(v)$), denoted by $d^+(v)$ (resp. $d^-(v)$) is called the out-degree (resp. in-degree) of $v$ in $D$. The degree of $v$ is $d(v)=d^+(v)+d^-(v)$. For a pair of vertex sets $X, Y$ of $D$, define $(X, Y)=\{xy\in A(D) : x\in X, y\in Y\}$. Let $\overleftrightarrow{a}(X, Y)=|(X,Y)|+|(Y,X)|$.

Let $P=y_0y_1\ldots y_k$ be a $(y_0, y_k)$-path of $D$. For $i\neq j$, $y_i, y_j\in V(P)$ we denote by $y_iPy_j$ the subpath of $P$ from $y_i$ to $y_j$. If $0 < i\le k$, then the predecessor of $y_i$ on $P$ is the vertex $y_{i-1}$ and is also denoted by $y^-_i$. If $0\le i<k$, then the successor of $y_i$ on $P$ is the vertex $y_{i+1}$ and is also denoted by $y^+_i$. A $k$-cycle is a cycle of order $k$. A cycle factor in $D$ is a collection of vertex-disjoint cycles $C_1, C_2, \ldots, C_t$ such that $V(C_1)\cup V(C_2)\cup \cdots \cup V(C_t)=V(D)$.

A digraph $D$ is said to be strongly connected or just strong, if for every pair of vertices $x, y$ of $D$, there is a path with endvertices $x$ and $y$. A digraph $D$ is called hamiltonian if it contains a hamiltonian cycle, i.e., a cycle that includes every vertex of $D$. A digraph $D$ is traceable if $D$ possesses a hamiltonian path. A digraph $D$ is semicomplete, if for every pair of vertices $x$, $y$ of $D$, $xy$ or $yx$ (or both) is in $D$.

A digraph $D$ is bipartite when $V(D)$ is a disjoint union of independent sets $V_1$ and $V_2$. It is called  balanced if $|V_1|=|V_2|$. A matching from $V_1$ to $V_2$ is an independent set of arcs with origin in $V_1$ and terminus in $V_2$ ($u_1u_2$ and $v_1v_2$ are independent arcs when $u_1\neq v_1$ and $u_2\neq v_2$). If $D$ is balanced, one says that such a matching is perfect if it consists of precisely $|V_1|$ arcs. If $D$ is bipartite and for every pair of vertices $x, y$ from distinct partite sets, $xy$ and $yx$ are in $D$, then $D$ is called complete bipartite.

The cycle problems for digraphs are one of the central problems in graph theory and its applications \cite {bang}. There are many degree or degree sum conditions for hamiltonicity in digraphs. The following result of Meyniel on the existence of hamiltonian cycles in digraphs is basic and famous.

\begin{theorem}\cite{meyniel}
 Let $D$ be a strong digraph on $n$ vertices where $n\ge 3$. If $d(x)+d(y)\ge 2n-1$ for all pairs of non-adjacent vertices $x, y$ in $D$, then $D$ is hamiltonian.
\end{theorem}

Recently, there is a renewed interest in various degree conditions for hamiltonicity in bipartite digraphs (see, e.g., \cite{adamus2017}, \cite{adamus2018}, \cite{adamus1}, \cite{adamus2}, \cite{darbinyan}, \cite{meszka}, \cite{wang}, \cite{wang2}, \cite{wang3}). In particular, In \cite{adamus2}, Adamus et al. gave a Meyniel-type sufficient condition for hamiltonicity of balanced bipartite digraphs.

\begin{definition} Let $D$ be a balanced bipartite digraph of order $2a$, where $a\ge 2$. For an integer $k$, we will say that $D$ satisfies the condition $M_k$ when $d(u)+d(v)\ge 3a+k$, for all pairs of non-adjacent vertices $u,v$.
\end{definition}

\begin{theorem}\label{strong and not strong hamiltonian}\cite{adamus2} Let $D$ be a balanced bipartite digraph on $2a$ vertices, where $a\ge 2$. Then $D$ is hamiltonian provided one of the following holds:
\begin{description}
  \item[(a)] $D$ satisfies the condition $M_1$, or
  \item[(b)] $D$ is strong and satisfies the condition $M_0$.
\end{description}
\end{theorem}

In Section 3, we reduce the bound in Theorem \ref{strong and not strong hamiltonian}(b) by 1 and prove that $D$ is either hamiltonian or isomorphic to a digraph in $\mathcal{H}_1$ or the digraph $H_2$, see Examples \ref{example} and \ref{example2} below. From this, we determine the extremal digraph of Theorem \ref{strong and not strong hamiltonian}(b). We also prove that a strong balanced bipartite digraph of order $2a$ satisfying the condition $M_{-1}$ is traceable. Our proofs are based on the arguments of \cite{adamus2}.

\begin{example}\label{example}
For an odd integer $a\ge 3$, let $\mathcal{H}_1$ be a set of bipartite digraphs. For any digraph $H_1$ in $\mathcal{H}_1$, let $V_1$ and $V_2$ be partite sets of $H_1$ such that $V_1$ (resp. $V_2$) is a disjoint union of $S, R$ (resp. $U, W$) with $|S|=|W|=\frac{a+1}{2}$, $|U|=|R|=\frac{a-1}{2}$ and $A(H_1)$ consists of the following arcs:
\begin{description}
  \item[(a)] $rw$ and $wr$, for all $r\in R$ and $w\in W$;
  \item[(b)] $us$ and $su$, for all $u\in U$ and $s\in S$;
  \item[(c)] $ws$, for all $w\in W$ and $s\in S$;
  \item[(d)] there exist $r\in R$ and $u\in U$ such that $ur\in A(H_1)$. For every $r\in R$, $d_{H_1[U]}(r)\geq \frac{a-3}{2}$ and for every $u\in U$, $d_{H_1[R]}(u)\geq \frac{a-3}{2}$.
\end{description}

Note that $H_1$ is strong and satisfies the condition $M_{-1}$, but since $|N^+(S)|=|U|<|S|$, there exists no perfect matching from $V_1$ to $V_2$. Thus, $H_1$ is non-hamiltonian.
\end{example}

\begin{example}\label{example2} Let $H_2$ be a bipartite digraph with partite sets $X=\{x_1, x_2, x_3\}$ and $Y=\{y_1, y_2, y_3\}$. The arc set $A(H_2)$ consist of  the following arcs $x_1y_2$, $y_2x_3$, $x_3y_3$, $y_3x_1$ and the following 2-cycles $x_2\leftrightarrow y_2$, $x_2\leftrightarrow y_3$, $y_1\leftrightarrow x_1$ and $y_1\leftrightarrow x_3$. Note that $H_2$ is strong and the degree of every vertex in $H_2$ is 4. Thus $H_2$ satisfies the condition $M_{-1}$ as $a=3$. Observe that $H_2$ is non-hamiltonian (see Figure 1).
\end{example}

\unitlength 1mm 
\linethickness{0.4pt}
\ifx\plotpoint\undefined\newsavebox{\plotpoint}\fi 
\begin{picture}(133,24)(40,70)
\put(100,88){\circle*{2}}
\put(116,88){\circle*{2}}
\put(132,88){\circle*{2}}
\put(100,75){\circle*{2}}
\put(116,75){\circle*{2}}
\put(132,75){\circle*{2}}
\put(114.61,88.15){\vector(1,0){.07}}\put(101.53,88.15){\vector(-1,0){.07}}\put(101.53,88.15){\line(1,0){13.081}}
\put(114.63,75.13){\vector(1,0){.07}}\put(101.25,75.13){\vector(-1,0){.07}}\put(101.25,75.13){\line(1,0){13.375}}
\put(130.5,74.88){\vector(1,0){.07}}\put(117.5,74.88){\vector(-1,0){.07}}\put(117.5,74.88){\line(1,0){13}}
\put(130.75,88.13){\vector(1,0){.07}}\put(117.38,88.13){\vector(-1,0){.07}}\put(117.38,88.13){\line(1,0){13.375}}
\put(99.88,76.25){\vector(0,1){10.25}}
\put(131.75,76){\vector(0,1){10.63}}
\put(130.63,76.38){\vector(3,-1){.07}}\multiput(100.63,87.13)(.094043887,-.03369906){319}{\line(1,0){.094043887}}
\put(101.5,76.38){\vector(-3,-1){.07}}\multiput(130.75,87.63)(-.08757485,-.033682635){334}{\line(-1,0){.08757485}}
\put(99.17,91.39){\makebox(0,0)[cc]{$x_1$}}
\put(116.14,91.39){\makebox(0,0)[cc]{$y_1$}}
\put(132.76,91.39){\makebox(0,0)[cc]{$x_3$}}
\put(99.17,71){\makebox(0,0)[cc]{$y_3$}}
\put(116.14,71){\makebox(0,0)[cc]{$x_2$}}
\put(132.76,71){\makebox(0,0)[cc]{$y_2$}}
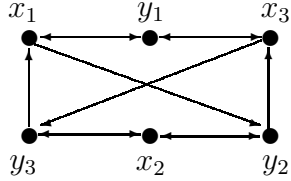
\put(114.73,65.58){\makebox(0,0)[cc]{Figure 1. the digraph $H_2$.}}
\end{picture}


\section{Lemmas}

\quad The proof of the main result will be based on the following several lemmas.

\begin{lemma}\label{cyclefactor}
  Let $D$ be a strong balanced bipartite digraph of order $2a$, where $a\ge 2$. If $D$ satisfies the condition $M_{-1}$, then either $D$ contains a cycle factor or $D$ is isomorphic to a digraph in $\mathcal{H}_1$.
\end{lemma}

\begin{proof}  Let $V_1$ and $V_2$ denote two partite sets of $D$. Observe that $D$ contains a cycle factor if and only if there exist both a perfect matching from $V_1$ to $V_2$ and a prefect matching from $V_2$ to $V_1$. In order to prove that $D$ contains a perfect matching from $V_1$ to $V_2$ and a prefect matching from $V_2$ to $V_1$, by the Hall theorem, it suffices to show that $|N^+(S)|\ge |S|$ for every $S\subset V_1$ and $|N^+(T)|\ge |T|$ for every $T\subset V_2$.

If there exists a non-empty set $S\subset V_1$ such that $|N^+(S)|<|S|$, then we will show that $D$ is isomorphic to a digraph in $\mathcal{H}_1$. Note that $V_2\setminus N^+(S)\neq\emptyset$. If $|S|=1$, write $S=\{x\}$, then $|N^+(S)|<|S|$ implies that $d^+(x)=0$. It is impossible in a strong digraph. Thus $|S|\ge 2$. If $|S|=a$, then every vertex from $V_2\setminus N^+(S)$ has in-degree zero, which again contradicts strong connectedness of $D$. Therefore, $2\le |S|\le a-1$.

For any $x_1, x_2\in S$ and $w_1, w_2\in V_2\setminus N^+(S)$, by the hypothesis of the lemma,
$$3a-1\le d(x_1)+d(x_2)\le 2|N^+(S)|+2a
\eqno(1)$$
and
$$3a-1\le d(w_1)+d(w_2)\le 2a+2(a-|S|).\eqno(2)$$

\noindent From these, we have $|N^+(S)|\ge \frac{a-1}{2}$ and $|S|\le \frac{a+1}{2}$.
If $a$ is even, then $|N^+(S)|\ge \frac{a}{2}$ and $|S|\le \frac{a}{2}$, which is a contradiction to $|N^+(S)|<|S|$. Thus $a$ is odd and $\frac{a-1}{2}\le |N^+(S)|\le |S|-1\le \frac{a+1}{2}-1=\frac{a-1}{2}$, which means $|N^+(S)|=\frac{a-1}{2}$ and $|S|=\frac{a+1}{2}$. Moveover, all equalities hold in (1) and (2), which means that $d^+(x_1)=d^+(x_2)=|N^+(S)|$, $d^-(x_1)=d^-(x_2)=a$, $d^-(w_1)=d^-(w_2)=a-|S|$ and $d^+(w_1)=d^+(w_2)=a$. By the strong connectedness of $D$ and the hypothesis of this lemma, $D$ is isomorphic to a digraph in $\mathcal{H}_1$.
\end{proof}

 From the proof of Theorem 1.2 in \cite{adamus2}, we have the following lemma. We provide its proof for completeness.

\begin{lemma}\label{cycle number}
Let $D$ be a bipartite digraph with partite sets $V_1$ and $V_2$. Suppose that $C_i$ and $C_j$ are two vertex-disjoint cycles in $D$. If $C_i$ and $C_j$ cannot be mergered into a cycle with vertex set $V(C_i)\cup V(C_j)$, then $\overleftrightarrow{a}(V(C_i), V(C_j))\le \frac{|V(C_i)|\cdot|V(C_j)|}{2}$. Moreover, if $\overleftrightarrow{a}(V(C_i), V(C_j))=\frac{|V(C_i)|\cdot|V(C_j)|}{2}$, then for any $x_i\in V(C_i)\cap V_q$ and $x_j\in V(C_j)\cap V_q$, $|\{x_ix_j^+, x_jx_i^+\}\cap A(D)|=1$, with $q\in \{1,2\}$.
\end{lemma}

\begin{proof} Let $q\in \{1,2\}, x_i\in V(C_i)\cap V_q$ and $x_j\in V(C_j)\cap V_q$ be arbitrary. Let $x_i^+$ be the successor of $x_i$ in $C_i$ and let $x_j^+$ be the successor of $x_j$ in $C_j$. Let $\mathcal{Z}_q(x_i, x_j)$ be defined as $\{x_ix_j^+, x_jx_i^+\}\cap A(D)$. If $|\mathcal{Z}_q(x_i,x_j)|=2$ for some $x_i, x_j$, then the cycles $C_i$ and $C_j$ can be merged into one cycle by deleting the
arcs $x_ix_i^+$ and $x_jx_j^+$ and adding the arcs $x_ix_j^+$ and $x_jx_i^+$, a contradiction. So we may assume that
$$|\mathcal{Z}_q(x_i, x_j)|\le 1,\  \mbox{for all}\  x_i\in V(C_i)\cap V_q \mbox{\ and \ } x_j\in V(C_j)\cap V_q. \eqno(1)$$
Now, consider an arc $uv\in (V(C_i), V(C_j))$ and assume $u\in V_q$. Let $v^-$ denote the predecessor of $v$ in $C_j$. Then $uv\in \mathcal{Z}_q(u,v^-)$. Similarly, if $uv\in (V(C_j), V(C_i))$, $u\in V_q$, and $v^-$ is the predecessor of $v$ in $C_i$, then $uv\in \mathcal{Z}_q(v^-, u)$. Therefore
$$\overleftrightarrow{a}(V(C_i),V(C_j))\le \sum\limits_{q=1}^{2}\sum\limits_{x_i\in V(C_i)\cap V_q}\sum\limits_{x_j\in V(C_j)\cap V_q} |\mathcal{Z}_q(x_i, x_j)|,$$
and hence, by (1),  $$\overleftrightarrow{a}(V(C_i),V(C_j))\le 2 \cdot  \frac{|V(C_i)|}{2}\cdot \frac{|V(C_j)|}{2}. $$

Moreover, if $\overleftrightarrow{a}(V(C_i), V(C_j))=\frac{|V(C_i)|\cdot|V(C_j)|}{2}$, then the equality holds in (1), that is to say, $|\{x_ix_j^+, x_jx_i^+\}\cap A(D)|=1$, which completes the proof of the lemma.
\end{proof}

The next lemma shows two simple results.

\begin{lemma}\label{number} Let $a_1, a_2, \ldots, a_t$ be non-negative integers with $a_1\le a_2\le \cdots \le a_t$ and let $A$ be a positive integer. If $a_1+a_2+\cdots+a_t\le A$, then the following hold.

\begin{description}
\item (a) For any $l\in \{1,2,\ldots,t\}$, $a_1+a_2+\cdots+a_l\le \frac{lA}{t}$;

\item (b) If $a_1+a_2=\frac{2A}{t}$, then for all $i\neq j \in \{1,2,\ldots,t\}$, we have $a_i=\frac{A}{t}$, $a_i+a_j=\frac{2A}{t}$ and $a_1+a_2+\cdots+a_t=A$.
\end{description}
\end{lemma}

\begin{proof}
  (a) For a proof by contradiction, suppose that $a_1+a_2+\cdots+a_l>\frac{lA}{t}$.  Then $a_l>\frac{A}{t}$, as otherwise $a_1\le a_2\le \cdots \le a_l\le\frac{A}{t}$ implies $a_1+a_2+\cdots+a_l\le \frac{lA}{t}$, a contradiction. Then $\frac{A}{t}<a_l\le a_{l+1}\le\cdots \le a_t$ implies $\frac{(t-l)A}{t}<a_{l+1}+\cdots+a_t\le A-(a_1+a_2+\cdots+a_l)<A-\frac{lA}{t}=\frac{(t-l)A}{t},$ a contradiction. Hence $a_1+a_2+\cdots+a_l\le \frac{lA}{t}$.

   (b) If $t=2$, then there is nothing to prove. Now assume $t\ge 3$. First $a_1=a_2=\frac{A}{t}$, as otherwise $a_2>\frac{A}{t}$ implies $a_i>\frac{A}{t}$, for all $i\ge 3$. Then
   $\frac{(t-2)A}{t}< a_3+\cdots+a_t\le A-(a_1+a_2)= \frac{(t-2)A}{t}$, a contradiction. So $a_i\ge\frac{A}{t}$, for all $i\ge 3$.  Then $\frac{(t-2)A}{t}\le  a_3+\cdots+a_t\le A-(a_1+a_2)= \frac{(t-2)A}{t}$. It follows that all equalities hold. Then $a_i=\frac{A}{t}$ for all $i \in \{1,2,\ldots,t\}$. So $a_i+a_j=\frac{2A}{t}$ and $a_1+a_2+\cdots+a_t=A$.
\end{proof}

\begin{theorem}\label{semicomplet bipartite cycle factor}\cite{gutin, haggkvist}
Let $D$ be a strong semicomplete bipartite digraph. If $D$ contains a cycle factor, then $D$ is hamiltonian.
\end{theorem}

\section{Proof of the main result}

\begin{theorem}\label{main result}
Let $D$ be a strong balanced bipartite digraph of order $2a$ with partite sets $V_1$ and $V_2$, where $a\ge 3$. If $D$ satisfies the condition $M_{-1}$, then $D$ is either hamiltonian or isomorphic to a digraph in $\mathcal{H}_1$ or the digraph $H_2$.
\end{theorem}
\begin{proof} Suppose that $D$ is not isomorphic to a digraph in $\mathcal{H}_1$. By Lemma \ref{cyclefactor}, $D$ contains a cycle factor $C_1, C_2, \ldots, C_s$. Assume that $s$ is minimum possible and $D$ is not hamiltonian. So $s\ge 2$. Without loss of generality, assume that $|V(C_1)|\le |V(C_2)|\le \cdots \le |V(C_s)|$.  Clearly, $|V(C_1)|\le a$.  Denote $\overline{C}_1=D-V(C_1)$. By Lemma \ref{cycle number}, the following holds:

\begin{align}\label{1}
&\overleftrightarrow{a}(V(C_1)\cap V_1,V(\overline{C}_1))+
\overleftrightarrow{a}(V(C_1)\cap V_2, V(\overline{C}_1))\nonumber\\
 &=\overleftrightarrow{a}(V(C_1),V(\overline{C}_1))=
\sum_{i=2}^{s} \overleftrightarrow{a}(V(C_1),V(C_i)) \\
&\le \frac{|V(C_1)|(2a-|V(C_1)|)}{2}.  \nonumber
\end{align}

Without loss of generality, we may assume that

\begin{align}\label{2}
\overleftrightarrow{a}(V(C_1)\cap V_1, V(\overline{C}_1))\le \frac{|V(C_1)|(2a-|V(C_1)|)}{4},
\end{align}
as otherwise
\begin{align}\label{3}
\overleftrightarrow{a}(V(C_1)\cap V_2, V(\overline{C}_1))\le \frac{|V(C_1)|(2a-|V(C_1)|)}{4}.
\end{align}

To complete the proof, we first give the following two claims.

\vskip 0.2cm

{\noindent\bf Claim 1.} For any two non-adjacent vertices $x$ and $y$, if $d(x)\le b$, then $d(y)\ge 3a-1-b$.

\begin{proof} By the hypothesis of this theorem, $d(x)+d(y)\ge 3a-1$. This together with $d(x)\le b$ implies $d(y)\ge 3a-1-b$.
\end{proof}

{\noindent\bf Claim 2.} If $s=2$ and $D[V(C_1)]$ is either a complete bipartite digraph, or a complete bipartite digraph  minus one arc with $|V(C_1)|\ge 6$, then there exists a vertex $z\in V(C_2)$ such that $d_{C_1}(z)=0$.

\begin{proof} Suppose, on the contrary, that for every $z\in V(C_2)$, $d_{C_1}(z)> 0$, where $d_{C_1}(z)=|N^+(z)\cap V(C_1)|+|N^-(z)\cap V(C_1)|$. Since $D$ is strong, there exist arcs from $C_2$ to $C_1$. Without loss of generality, assume that $v\rightarrow x$, where $v\in V(C_2)\cap V_2$ and $x\in V(C_1)\cap V_1$. Let $y$ be an arbitrary vertex in $V(C_1)\cap V_2$.

First, we observe that there exists a hamiltonian path $Q$ from $x$ to $y$ in $D[V(C_1)]$. If $D[V(C_1)]$ is a complete bipartite digraph, it is obvious. Assume that $D[V(C_1)]$ is a complete bipartite digraph minus one arcs, say $e$, with $|V(C_1)|\ge 6$. Denote $m=\frac{|V(C_1)|}{2}$. Let $x_iy_i$, $i=1,2,\ldots, m$, be a perfect matching from $V_1$ to $V_2$ in $D[V(C_1)]$. Without loss of generality, assume $x=x_1$ and $y=y_m$ (This is possible as there are at least $m-1$ arc-disjoint perfect matchings from $V_1$ to $V_2$ in $D[V(C_1)]$ and $m\ge 3$). Let $P=x_1y_1x_2y_2\ldots x_my_m$. If $e\notin \{y_ix_{i+1}: i= 1,2, \ldots, m-1\}$, then $P$ is the desired path. If $e\in \{y_ix_{i+1}: i= 1,2, \ldots, m-1\}$, say $e=y_rx_{r+1}$, then $P\cup \{x_1y_r, y_rx_2, x_ry_1, y_1x_{r+1}\}\setminus \{x_1y_1, y_1x_2, x_ry_r, y_rx_{r+1}\}$ is the desired path.

According to the above observation, we can deduce that $y\nrightarrow v^+$, where $v^+$ is the successor of $v$ in $C_2$, otherwise $vxQyv^+C_2v$ is a hamiltonian cycle of $D$, a contradiction. By the arbitrariness of $y$, this means that $d_{C_1}^-(v^+)=0$. By $d_{C_1}(v^+)>0$, we have $d_{C_1}^+(v^+)>0$. Similarly, we can obtain that $d_{C_1}^-(w)=0$, for every $w\in V(C_2)$, a contradiction to the fact that $D$ is strong. The proof of the claim is complete.
\end{proof}

We now consider the following two cases.

\vskip 0.2cm
{\noindent\bf Case 1.} $|V(C_1)|=2$.
\vskip 0.2cm

Let $V(C_1)\cap V_1=\{x_1\}$ and $V(C_1)\cap V_2=\{y_1\}$.
By (\ref{1}),
       \begin{align}\label{4}
          d_{\overline{C}_1}(x_1)+d_{\overline{C}_1}(y_1)\le 2a-2
       \end{align}
\noindent and by (\ref{2}),
       \begin{align}
         d_{\overline{C}_1}(x_1)\le a-1.
       \end{align}

\noindent So $d(x_1)=d_{C_1}(x_1)+d_{\overline{C}_1}(x_1)\le a+1$.

Assume $d(x_1)\le a$. By Claim 1, $d(z)\ge 3a-1-a=2a-1$, for any $z\in V(D)$ such that $z$ and $x_1$ are non-adjacent. It is easy to see that $D$ is a semicomplete bipartite digraph. By Theorem \ref{semicomplet bipartite cycle factor}, $D$ is hamiltonian, a contradiction.

Now assume $d(x_1)=a+1$. By Claim 1, for any $x'\in V_1\setminus \{x_1\}$, $d(x')\ge 3a-1-(a+1)=2a-2$. By (\ref{4}) and $d_{\overline{C}_1}(x_1)=a-1$, $d_{\overline{C}_1}(y_1)\le a-1$ and so $d(y_1)\le a+1$. Similarly, we can also obtain $d(y_1)=a+1$. Hence, for any $y'\in V_2\setminus\{y_1\}$, $d(y')\ge (3a-1)-(a+1)=2a-2$. In fact, we have shown that
for any $w\in V(D)\setminus \{x_1,y_1\}$, $d(w)\ge 2a-2$.

Assume that $|V(C_2)|=2$. Write $C_2=x_2y_2x_2$, where $x_2\in V_1$ and $y_2\in V_2$. Analogously, we can also obtain that $d(x_2)=d(y_2)=a+1$. Note that $d(x_2)\ge 2a-2$. Thus, $2a-2\le d(x_2)=a+1$. From this, we have $a\le 3$ and so $a=3$ and $s=3$. Write $C_3=x_3y_3x_3$, where $x_3\in V_1$ and $y_3\in V_2$. Analogously, we can also obtain that $d(x_3)=d(y_3)=a+1$. By Theorem \ref{semicomplet bipartite cycle factor}, $D$ is not a semicomplete bipartite digraph. Hence, there exist two vertices from different partite sets such that they are not adjacent. Without loss of generality, assume that $x_1$ and $y_2$ are not adjacent. By $d(x_1)=a+1=4$ and $d(y_2)=a+1=4$, we have that $x_1\leftrightarrow y_3$ and $x_3\leftrightarrow y_2$. By Lemma \ref{cycle number}, $\overleftrightarrow{a}(V(C_2), V(C_3))\le 2$. So $x_2$ and $y_3$ are not adjacent. Then $d(x_2)=a+1=4$ implies that $x_2\leftrightarrow y_1$.  Note that $D$ is hamiltonian, a contradiction (see Figure 2).

\unitlength 1mm 
\linethickness{0.4pt}
\ifx\plotpoint\undefined\newsavebox{\plotpoint}\fi 
\begin{picture}(140,32)(35,70)
\put(84.5,101){\circle*{2}}
\put(97.5,101){\circle*{2}}
\put(84.5,91){\circle*{2}}
\put(97.5,91){\circle*{2}}
\put(84.5,81){\circle*{2}}
\put(97.5,81){\circle*{2}}
\put(96.5,101.25){\vector(1,0){.07}}\put(85.25,101.25){\vector(-1,0){.07}}\put(85.25,101.25){\line(1,0){11.25}}
\put(96,91){\vector(1,0){.07}}\put(85.5,91){\vector(-1,0){.07}}\put(85.5,91){\line(1,0){10.5}}
\put(96,81.5){\vector(1,0){.07}}\put(85.75,81.5){\vector(-1,0){.07}}\put(85.75,81.5){\line(1,0){10.25}}
\put(96.75,100){\vector(3,2){.07}}\put(85.25,92){\vector(-3,-2){.07}}\multiput(85.25,92)(.04831933,.03361345){238}{\line(1,0){.04831933}}
\put(96.25,89.75){\vector(3,2){.07}}\put(85.5,82.25){\vector(-3,-2){.07}}\multiput(85.5,82.25)(.04820628,.03363229){223}{\line(1,0){.04820628}}
\put(96.75,82.5){\vector(2,-3){.07}}\put(85.5,99.75){\vector(-2,3){.07}}\multiput(85.5,99.75)(.033682635,-.051646707){334}{\line(0,-1){.051646707}}
\put(80.5,101){\makebox(0,0)[cc]{$x_1$}}
\put(101.5,101){\makebox(0,0)[cc]{$y_1$}}
\put(80.5,91){\makebox(0,0)[cc]{$x_2$}}
\put(101.5,91){\makebox(0,0)[cc]{$y_2$}}
\put(80.5,81){\makebox(0,0)[cc]{$x_3$}}
\put(101.5,81){\makebox(0,0)[cc]{$y_3$}}
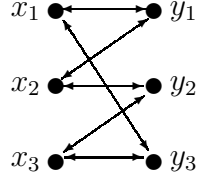
\put(90,73.25){\makebox(0,0)[cc]{Figure 2. The case when $|V(C_2)|=2$.}}
\end{picture}

Next assume that $|V(C_2)|\ge 4$. From this, $|V(C_i)|\ge 4$, for $i=3,\ldots, s$.  Let $D'=D-\{x_1, y_1\}$ and $a'=a-1$. First we claim that $s=2$. It suffices to show that $D'$ is hamiltonian. For $a'=2$ and $a'=3$, it is obvious. Recall that for any $u\in V(D')$, $d_D(u)\ge 2a-2$. Thus, $d_{D'}(u)\ge 2a-4=2a'-2$. Thus, for any two non-adjacent vertices $u$ and $v$ in $D'$, $d_{D'}(u)+d_{D'}(v)\ge 2(2a'-2)$. If $a'\ge 5$, then $2(2a'-2)\ge 3a'+1$. By Theorem \ref{strong and not strong hamiltonian}(a), $D'$ is hamiltonian. If $a'=4$, then $2(2a'-2)\ge 3a'$.  If $D'$ is strong, then by Theorem \ref{strong and not strong hamiltonian}(b), $D'$ is hamiltonian. Next assume that $D'$ is not strong. In this case, $s=3$ and $C_2, C_3$ are both 4-cycles. Write $C_2=x_2y_2x_3y_3x_2$, where $x_i\in V_1$ and $y_i\in V_2$, for $i=2,3$. Since $D'$ is not strong, without loss of generality, assume that $C_2\Rightarrow C_3$. So $d_{D'}(x_2)\le 2+4=2a'-2$ and $d_{D'}(y_2)\le 2a'-2$. Combining this with $d_{D'}(x_2)\ge 2a'-2$ and $d_{D'}(y_2)\ge 2a'-2$, we have that $d_{D'}(x_2)= d_{D'}(y_2)=2a'-2=2a-4$. Recall that $d_D(x_2)\ge 2a-2$ and $d_D(y_2)\ge 2a-2$.
So $x_2\leftrightarrow y_1$ and $x_1\leftrightarrow y_2$. This means that $C_1$ can be merged with $C_2$ by replacing the arc $x_2y_2$ on $C_2$ with the path $x_2y_1x_1y_2$, a contradiction. Hence $s=2$. Write $C_2=x_2y_2\ldots x_ay_ax_2$, where $x_i\in V_1$ and $y_i\in V_2$, for $i=2,\ldots, a$. By Claim 2, there exists a vertex $z\in V(C_2)$ such that $d_{C_1}(z)=0$, say $x_2$. Thus, $x_2$ and $y_1$ are not adjacent and $d(x_2)\le 2a-2$. From this with $d(x_2)\ge 2a-2$, we have that $d(x_2)=2a-2$, which implies that $x_2\leftrightarrow y_i$, for $i=2,\ldots, a$.  Recalling that $d_{C_2}(x_1)=d_{C_2}(y_1)=a-1$, that is to say, $\overleftrightarrow{a}(V(C_1), V(C_2))=2(a-1)=\frac{|V(C_1)|\cdot|V(C_2)|}{2}$.
 By Lemma \ref{cycle number}, for any $x_i, y_i\in V(C_2)$,
\begin{align}\label{6}
|\{x_iy_1, x_1y_i\}\cap A(D)|=1  \mbox{\ and\ } |\{y_{i-1}x_1,y_1x_i\}\cap A(D)|=1.
\end{align}
\noindent Since $y_1$ and $x_2$ are not adjacent, by (6), we have $y_a\rightarrow x_1$ and $x_1\rightarrow y_2$.

First consider the case when $a=3$. By $d_{C_2}(x_1)=d_{C_2}(y_1)=a-1$, we have $x_1\mapsto y_2$ and $y_3\mapsto x_1$ and $y_1\leftrightarrow x_3$. If $x_3\rightarrow y_2$, then $x_3y_2x_2y_3x_1y_1x_3$ is a hamiltonian cycle, a contradiction.  Hence, $y_2\mapsto x_3$. If $y_3\rightarrow x_3$, then $y_3x_3y_1x_1y_2x_2y_3$ is a hamiltonian cycle, a contradiction.
Hence $x_3\mapsto y_3$. Then $D$ is isomorphic to the digraph $H_2$ (see Figure 1.)

Next consider the case when $a\ge 4$. Assume that $x_a\rightarrow y_1$.
By (6), $x_1\nrightarrow y_a$. Furthermore, $y_a\nrightarrow x_3$, otherwise $x_1y_2x_2y_ax_3C_2x_ay_1x_1$ is a hamiltonian cycle, a contradiction. Hence $d(y_a)\le 2a-2$. Combining this with $d(y_a)\ge 2a-2$, we have $d(y_a)=2a-2$, which implies $x_3\mapsto y_a$ and $y_a\leftrightarrow x_i$, for $i=4,\ldots, a$. Moreover, $y_1\nrightarrow x_3$, otherwise $y_1x_3y_ax_1y_2x_2y_3C_2x_ay_1$ is a hamiltonian cycle, a contradiction. From this, we see that $d(x_3)\le 2a-2$. Combining this with $d(x_3)\ge 2a-2$, we have $d(x_3)=2a-2$, which implies $y_3\leftrightarrow x_3$. However, $x_3y_ax_4C_2x_ay_1x_1y_2x_2y_3x_3$ is a hamiltonian cycle, a contradiction. Now we assume $x_a\nrightarrow y_1$. Since $d_{C_2}(y_1)=a-1$ and $y_1$ and $x_2$ are not adjacent, there exists a vertex $x_i\in \{x_3,\ldots, x_{a-1}\}$ such that $x_i\rightarrow y_1$. Take $r=\max\{i: i\in \{3, \ldots, a-1\}\ \mbox{and} \ x_i\rightarrow y_1\}$. By the choose of $r$, for every $j\in \{r+1,\ldots, a\}$, $x_j\nrightarrow y_1$. Then by (6), $x_1\rightarrow y_j$. If $x_j\rightarrow y_2$, then $x_ry_1x_1y_jC_2x_2y_rC_2x_jy_2C_2x_r$ is a hamiltonian cycle, a contradiction. Hence $x_j\nrightarrow y_2$. Combining this with $x_j\nrightarrow y_1$ and $d(x_j)\ge 2a-2$, we have $d(x_j)=2a-2$. Hence $x_j\rightarrow \{y_j,y_{j-1}\}\rightarrow x_j$. But $x_ry_1x_1y_ax_ay_{a-1}\ldots y_rx_2C_2x_r$ is a hamiltonian cycle, a contradiction.

\vskip 0.2cm
{\noindent\bf Case 2.} $|V(C_1)|\ge 4$.
\vskip 0.2cm

In this case, $a\ge 4$. Let $x_1, x_2\in V(C_1)\cap V_1$ be distinct and chosen so that $\overleftrightarrow{a}(\{x_1, x_2\}, V(\overline{C}_1))$ is minimum. By Lemma \ref{number}(a) and (\ref{2}), $\overleftrightarrow{a}(\{x_1,x_2\}, V(\overline{C}_1))\le 2a-|V(C_1)|$, that is to say, $d_{\overline{C}_1}(x_1)+d_{\overline{C}_1}(x_2)\le 2a-|V(C_1)|$. Since any vertex in $C_1$ has at most $|V(C_1)|$ arcs to other vertices in $C_1$ (as there are $\frac{|V(C_1)|}{2}$ vertices from $V_2$ in $C_1$) and $|V(C_1)|\le a$, we get $3a-1\le d(x_1)+d(x_2)\le 2a+|V(C_1)|\le 3a.$ From this $|V(C_1)|=a-1$ or $|V(C_1)|=a$. Before we consider these two cases, we claim the following. Clearly, $s=2$.

%
%
%
%

\vskip 0.2cm
{\noindent\bf Claim 3.} For any $u\in V(C_2)$, $d_{C_1}(u)>0$.

\begin{proof}
  Suppose, on the contrary, that there exists $u_0\in V(C_2)$ such that $d_{C_1}(u_0)=0$.  Then $3a-1\le d(u_0)+d(x_i)\le |V(C_2)|+|V(C_1)|+d_{C_2}(x_i)$, for $i=1,2$. From this, $d_{C_2}(x_i)\ge a-1$. Thus $2(a-1)\le d_{C_2}(x_1)+d_{C_2}(x_2)\le 2a-|V(C_1)|$, which means $|V(C_1)|\le 2$, a contradiction.
\end{proof}

From Claims 2 and $3$, we know that $D[V(C_1)]$ is not a complete bipartite digraph. Let $y_1, y_2\in V(C_1)\cap V_2$ be distinct and chosen such that $\overleftrightarrow{a}(\{y_1, y_2\}, V(C_2))$ is the minimum.

\vskip 0.2cm

{\noindent\bf Claim 4.} If $d_{C_2}(x_1)+d_{C_2}(x_2)=2a-|V(C_1)|$, then $d_{C_2}(y_1)+d_{C_2}(y_2)\le 2a-|V(C_1)|$.

\begin{proof}
If $d_{C_2}(x_1)+d_{C_2}(x_2)=2a-|V(C_1)|$, then by Lemma \ref{number}(b) and (2), $\overleftrightarrow{a} (V(C_1)\cap V_1, V(C_2))=\frac{|V(C_1)| \cdot(2a-|V(C_1)|)}{4}$. Then by (1), $\overleftrightarrow{a} (V(C_1)\cap V_2, V(C_2))\le \frac{|V(C_1)|\cdot (2a-|V(C_1)|)}{4}$. By Lemma \ref{number}(a), $\overleftrightarrow{a} (\{y_1,y_2\}, V(C_2))\le 2a-|V(C_1)|$, that is, $d_{C_2}(y_1)+d_{C_2}(y_2)\le 2a-|V(C_1)|$.
\end{proof}

Now we return to the proof of the theorem and consider the following subcases.

\vskip 0.2cm
{\noindent \bf Subcase 2.1.} $|V(C_1)|=a-1$.
\vskip 0.2cm

In this case, $|V(C_1)|=a-1\ge 4$, that is $a\ge 5$, and $|V(C_2)|=a+1$.

\vskip 0.2cm
{\noindent\bf Claim 5.} For any two non-adjacent vertices $u, v\in V(C_1)$, if $d_{C_2}(u)+d_{C_2}(v)\le 2a-|V(C_1)|$, then $d_{C_2}(u)+d_{C_2}(v)= 2a-|V(C_1)|$ and $d_{C_1}(u)=d_{C_1}(v)=|V(C_1)|$.

\begin{proof}
  By hypothesis,  $3a-1\le d(u)+d(v)=d_{C_2}(u)+d_{C_2}(v)+d_{C_1}(u)+d_{C_1}(v)
\le 2a-|V(C_1)|+2|V(C_1)|=3a-1$. If follows that $d_{C_2}(u)+d_{C_2}(v)= 2a-|V(C_1)|$ and $d_{C_1}(u)=d_{C_1}(v)=|V(C_1)|$.
\end{proof}

By Claim 5, $d_{C_2}(x_1)+d_{C_2}(x_2)= 2a-|V(C_1)|$ and $d_{C_1}(x_1)=d_{C_1}(x_2)=|V(C_1)|$. By (\ref{2}) and Lemma \ref{number}(b),  for any $x', x''\in V(C_1)\cap V_1$, $d_{C_2}(x')+d_{C_2}(x'')=2a-|V(C_1)|$. By Claim 5, $d_{C_1}(x')=d_{C_1}(x'')=|V(C_1)|$. Then $D[V(C_1)]$ is a complete bipartite digraph, a contradiction.

\vskip 0.2cm
{\noindent \bf Subcase 2.2.} $|V(C_1)|=a$.
\vskip 0.2cm

In this case, $|V(C_2)|=a$. By
\begin{align}\label{7}
  3a-1& \le d(x_1)+d(x_2)\nonumber\\
      &  =  d_{C_1}(x_1)+d_{C_1}(x_2)+d_{C_2}(x_1)+d_{C_2}(x_2)\\
      & \le 2a+d_{C_2}(x_1)+d_{C_2}(x_2), \nonumber
\end{align}
\noindent we have $d_{C_2}(x_1)+d_{C_2}(x_2)\ge a-1$. Combining this with $d_{C_2}(x_1)+d_{C_2}(x_2)\le 2a-|V(C_1)|=a$, we have $d_ {C_2}(x_1)+d_{C_2}(x_2)=a-1$ or $d_{C_2}(x_1)+d_{C_2}(x_2)=a$.

First suppose $d_{C_2}(x_1)+d_{C_2}(x_2)=2a-|V(C_1)|=a$. By Lemma \ref{number}(b) and (2), $d_{C_2}(x_i)+d_{C_2}(x_j)=a$, for any $x_i, x_j\in V(C_1)\cap V_1$. Since $D[V(C_1)]$ is not a complete bipartite digraph, there exists a vertex $x'\in V(C_1)\cap V_1$ such that $d_{C_1}(x')\le a-1$. For any $x_k\in (V(C_1)\cap V_1)\setminus\{x'\}$, $3a-1\le d(x')+d(x_k)=(d_{C_1}(x')+d_{C_1}(x_k))+(d_{C_2}(x')+d_{C_2}(x_k))\le (a-1+a)+a=3a-1$. So $d_{C_1}(x')=a-1$ and $d_{C_1}(x_k)=a$, which implies that $D[V(C_1)]$ is a complete bipartite digraph minus one arc. According to Claims 2 and 3, $|V(C_1)|=4$. Write $C_1=x_1y_1x_2y_2x_1$ and $C_2=x_3y_3x_4y_4x_3$, where $x_i\in V_1$ and $y_i\in V_2$, for $i=1,2, 3, 4$. Without loss of generality, assume that $d_{C_1}(x_1)=3$ and $y_2\mapsto x_1$.  According to Claim 4, $d_{C_2}(y_1)+d_{C_2}(y_2)\le a$. Then $3a-1\le d(y_1)+d(y_2)\le 2a-1+a=3a-1$ implies that $d_{C_2}(y_1)+d_{C_2}(y_2)=a$, which means $d_{C_1}(x_3)+d_{C_1}(x_4)=a$. By symmetry, we can deduce that $D[V(C_2)]$ is a complete bipartite digraph minus one arc. Without loss of generality, assume that $d_{C_2}(x_3)=a-1$. Then $3a-1\le d(x_1)+d(x_3)\le 2(a-1)+d_{C_2}(x_1)+d_{C_1}(x_3)$, that is, $d_{C_2}(x_1)+d_{C_1}(x_3)\ge a+1$. Without loss of generality, assume that $d_{C_2}(x_1)\ge \frac{a}{2}+1$. Combining this with $d_{C_2}(x_1)+d_{C_2}(x_2)=a$, we have $d_{C_2}(x_2)\le \frac{a}{2}-1$. Then $3a-1\le d(x_2)+d(x_3)\le (a+\frac{a}{2}-1)+(a-1+d_{C_1}(x_3))$ implies that $d_{C_1}(x_3)\ge \frac{a}{2}+1$. From this with $d_{C_1}(x_3)+d_{C_1}(x_4)=a$, we have $d_{C_1}(x_4)\le \frac{a}{2}-1$. But $d(x_2)+d(x_4)\le 2a+2(\frac{a}{2}-1)=3a-2$, a contradiction.


Now suppose $d_{C_2}(x_1)+d_{C_2}(x_2)=a-1$. From (\ref{7}), $d_{C_1}(x_1)=d_{C_1}(x_2)=a$. If $a=4$, then $D[V(C_1)]$ is a complete bipartite digraph, a contradiction. Next assume that $a\ge 6$. By Claims 2 and 3, $D[V(C_1)]$ is not a complete bipartite digraph minus one arc. Denote $V(C_1)\cap V_1=\{x_1, x_2, \ldots, x_{\frac{a}{2}}\}$ and without loss of generality, assume that $d_{C_2}(x_1)\le d_{C_2}(x_2)\le \cdots \le d_{C_2}(x_{\frac{a}{2}})$.  By the choice of $x_1$ and $x_2$ and $d_{C_2}(x_1)+d_{C_2}(x_2)=a-1$, we know that $d_{C_2}(x_1)\le \frac{a}{2}-1$. Denote $d_{C_2}(x_1)=\frac{a}{2}-k$, with $k\ge 1$. So $d_{C_2}(x_2)=\frac{a}{2}+k-1$ and $d_{C_2}(x_i)\ge \frac{a}{2}+k-1$, for $i=3, \ldots, \frac{a}{2}$. By (\ref{2}),
\begin{align}\label{9}
  d_{C_2}(x_1)+d_{C_2}(x_2)+\cdots +d_{C_2}(x_{\frac{a}{2}})\le \frac{a^2}{4}.
\end{align}
Since $D[V(C_1)]$ is neither a complete bipartite digraph nor a complete bipartite digraph minus one arc, either there exists a vertex $x_i\in V(C_1)\cap V_1$ such that $d_{C_1}(x_i)\le a-2$ or there exist at least two vertices $x_i$ and $x_j$ such that $d_{C_1}(x_i)=a-1$ and $d_{C_1}(x_j)=a-1$. If $a=6$, then $d_{C_1}(x_3)\le a-2$. By $3a-1\le d(x_1)+d(x_3)\le (a+\frac{a}{2}-k)+(a-2+d_{C_2}(x_3))$, we have $d_{C_2}(x_3)\ge \frac{a}{2}+k+1$. So $d_{C_2}(x_1)+d_{C_2}(x_2)+d_{C_2}(x_3)\ge a-1+\frac{a}{2}+k+1=\frac{3a}{2}+k$. According to (\ref{9}), $\frac{3a}{2}+k\le \frac{a^2}{4}$. It is impossible as $k\ge 1$ and $a=6$, a contradiction.
Hence $a\ge 8$.
By $\frac{a}{2}-k+(\frac{a}{2}-1)(\frac{a}{2}+k-1)\le \sum_{i=1}^{\frac{a}{2}} d_{C_2}(x_i)\le \frac{a^2}{4}$, we have $k\le 1$ and so $k=1$.
So $d_{C_2}(x_1)=\frac{a}{2}-1$ and $d_{C_2}(x_i)\ge \frac{a}{2}$, for $i\ge 2$.
Suppose that there exists a vertex $x'\in V(C_1)\cap V_1$ such that $d_{C_1}(x')\le a-2$.
By Claim 1 and $d(x_1)=d_{C_1}(x_1)+d_{C_2}(x_1)=a+\frac{a}{2}-1=\frac{3a}{2}-1$, we have $d(x')\ge \frac{3a}{2}$ and so $d_{C_2}(x')=d(x')-d_{C_1}(x')\ge \frac{a}{2}+2$.
Then $\frac{a^2}{4}+1=(\frac{a}{2}-1)+(\frac{a}{2}-2)\frac{a}{2}+\frac{a}{2}+2\le \sum_{i=1}^{\frac{a}{2}}d_{C_2}(x_i)\le \frac{a^2}{4}$, a contradiction.
Thus there exist two vertices $x_i, x_j\in V(C_1)\cap V_1$ such that $d_{C_1}(x_i)=a-1$ and $d_{C_1}(x_j)=a-1$. Then by Claim 1, $d_{C_2}(x_i)\ge (3a-1)-d(x_1)-d_{C_1}(x_i)=(3a-1)-(\frac{3a}{2}-1)-(a-1)=\frac{a}{2}+1$ and $d_{C_2}(x_j)\ge \frac{a}{2}+1$. Then $\frac{a^2}{4}+1=(\frac{a}{2}-1)+(\frac{a}{2}-3)\frac{a}{2}+2(\frac{a}{2}+1)\le \sum\limits_{i=1}^{\frac{a}{2}}d_{C_2}(x_i)\le \frac{a^2}{4}$, a contradiction. We have considered all cases and completed the proof of the theorem.
\end{proof}

From Theorem \ref{main result}, we can obtain the following.

\begin{theorem}\label{hamiltonian path}
Let $D$ be a strong balanced bipartite digraph of order $2a$, where $a\ge 3$. If $D$ satisfies the condition $M_{-1}$, then $D$ is traceable.
\end{theorem}

\begin{proof} By Theorem \ref{main result}, $D$ is either hamiltonian or isomorphic to a digraph in $\mathcal{H}_1$ or the digraph $H_2$. If $D$ is hamiltonian, there is nothing to prove. If $D$ is isomorphic to the digraph $H_2$ (see Figure 1), then $x_1y_1x_3y_3x_2y_2$ is a hamiltonian path. Suppose that $D$ is isomorphic to a digraph in $\mathcal{H}_1$ (see Example \ref{example}). Note that both $D[S\cap U]$ and $D[R\cap W]$ are complete bipartite digraphs and $|S|=|W|=\frac{a+1}{2}\ge 2$. Clearly, for any $x_1, x_2\in S$, there is a hamiltonian path $Q_1$ from $x_1$ to $x_2$ in $D[S\cap U]$ and for any $w_1, w_2\in W$, there is a hamiltonian path $Q_2$ from $w_1$ to $w_2$ in $D[R\cap W]$. Then $w_1Q_2w_2x_1Q_1x_2$ is a hamiltonian path in $D$.
\end{proof}

%
%

The bound in Theorem \ref{hamiltonian path} is sharp, as can be seen in the following example.

\begin{example}\label{D2} Let $a\geq 4$ be an even integer and let $H_3$ be a balanced bipartite digraph with partite sets $V_1$ and $V_2$ such that $V_1$ (resp. $V_2$) is a disjoint union of $S,R$ (resp. $U,W$) with $|S|=|W|=\frac{a+2}{2}$, $|U|=|R|=\frac{a-2}{2}$, and $A(H_3)$ consists of the following arcs:
\begin{itemize}
 \item [\bf ($a$)] $ry$ and $yr$, for all $r\in R$ and $y\in V_2$,
 \item [\bf ($b$)] $ux$ and $xu$, for all $u\in U$ and $x\in V_1$, and
 \item [\bf ($c$)] $ws$,          for all $w\in W$ and $s\in S$.
\end{itemize}
\end{example}
Then $d(r)=d(u)=2a$ for all $r\in R$ and $u\in U$, and $d(s)=d(w)=\frac{3a-2}{2}$ for all $s\in S$ and $w\in W$ and so $H_3$ satisfies the condition $M_{-2}$. Notice that $H_3$ is strong, but contains no hamiltonian path, as the size of a maximum matching from $V_1$ to $V_2$ is $a-2$.

\section{Related problems}

\quad Let $\mathcal{D}_{a,k}$ denote all strong balanced bipartite digraphs on $2a$ vertices such that $d(u)+d(v)\ge 3a-k$ for all non-adjacent vertices $u,v$.  If $D\in \mathcal{D}_{a,0}$, then by Theorem \ref{strong and not strong hamiltonian},  $D$ is hamiltonian. A hamiltonian digraph must possess a cycle factor. In this present paper, we have shown that if $D\in \mathcal{D}_{a,1}$ and $D$ contains a cycle factor, then $D$ is hamiltonian unless $D$ is the digraph $H_2$. A natural question would be if there are at most a finite number (depending on $k$) of digraphs in $D\in \mathcal{D}_{a,k}$ containing a cycle factor but not a hamiltonian cycle.

Theorem \ref{semicomplet bipartite cycle factor} implies that a strong bipartite tournament containing a cycle factor, is hamiltonian. Let $D$ be a balanced bipartite oriented graph of order $2a$. An another natural question would be if there exists an integer $k\ge 1$ such that $D$ satisfying the condition $d(x)+d(y)\ge 2a-k$ for any pair of non-adjacent vertices $x, y$ in $D$  and containing a cycle factor, is hamiltonian.

 To conclude the paper, we mention two related problems. In \cite{bang2},  Bang-Jensen et al. conjectured the following strengthening of a classical Meyniel theorem.

\begin{conjecture}\cite{bang2}
  If $D$ is a strong digraph on $n$ vertices in which $d(u)+d(v)\ge 2n-1$ for every pair of non-adjacent vertices $u, v$ with a common out-neighbour or a common in-neighbour, then $D$ is hamiltonian.
\end{conjecture}

In \cite{adamus2017}, Adamus proved a bipartite analogue of the conjecture.

\begin{theorem}\cite{adamus2017}\label{51}
Let $D$ be a strong balanced bipartite digraph of order $2a$ with $a\ge 3$. If $d(x)+d(y)\ge 3a$ for every pair of vertices $x, y$ with a common out-neighbour or a common in-neighbour, then $D$ is hamiltonian.
\end{theorem}

A natural problem is to characterize the extremal digraph on the condition in Theorem \ref{51}.

A balanced bipartite digraph containing cycles of all even length is called bipancyclic. In \cite{adamus2018}, Adamus proved that the hypothesis of  Theorem \ref{51} implies bipancyclicity of $D$, except for a directed cycle of length $2a$ (Theorem \ref{52} below).

\begin{theorem}\cite{adamus2018}\label{52}
Let $D$ be a strong balanced bipartite digraph of order $2a$ with $a\ge 3$. If $d(x)+d(y)\ge 3a$ for every pair of vertices $x, y$ with a common out-neighbour or a common in-neighbour in $D$, then $D$ either is bipancyclic or is a directed cycle of length $2a$.
\end{theorem}

In the same paper, the author presented the following problem: if for every $1\le l<a$ there is an interger $k\ge 1$ such that every strong balanced bipartite digraph on $2a$ vertices contains cycles of all even lengths up to $2l$, provided $d(x)+d(y)\ge 3a-k$ for every pair of vertices $x,y$ with a common out-neighbour or a common in-neighbour.

\acknowledgements

\quad We would like to thank the anonymous referees for many valuable suggestions that allowed us to improve the exposition of our results.



\end{document}